\documentclass[a4paper, 12pt]{amsart}
\usepackage[active]{srcltx}
\usepackage[all]{xy}
\usepackage{subfig}
\usepackage{hyperref}
\hypersetup{
colorlinks,
allcolors=blue
}
\usepackage{mathrsfs}

\usepackage{amsmath,amssymb,amsthm,a4wide}

\usepackage{enumerate}
\usepackage{color}

\theoremstyle{plain}
\newtheorem{lemma}{Lemma}[section]
\newtheorem{theorem}[lemma]{Theorem}

\newtheorem{corollary}[lemma]{Corollary}
\newtheorem{conjecture}[lemma]{Conjecture}

\theoremstyle{definition}

\newtheorem{definition}[lemma]{Definition}
\newtheorem{remark}[lemma]{Remark}

\theoremstyle{remark}

\numberwithin{equation}{section}

\newcommand{\R}{\mathbb{R}}

\DeclareMathOperator{\0}{\mathbf{0}}

\newcommand{\be}{\begin{equation}}
\newcommand{\ee}{\end{equation}}

\numberwithin{equation}{section}

\usepackage{listings}
\usepackage{xcolor}

\makeatletter
\def\@settitle{\begin{center}%
  \baselineskip15\p@\relax
  \bfseries \large
  \@title
  \end{center}%
}

\def\@setauthors{%
  \begingroup
  \def\thanks{\protect\thanks@warning}%
  \trivlist
  \centering\normalsize \@topsep30\p@\relax 
  \advance\@topsep by -\baselineskip
  \item\relax
  \author@andify\authors
  \def\\{\protect\linebreak}%
  \authors 
  \ifx\@empty\contribs
  \else
    ,\penalty-3 \space \@setcontribs
    \@closetoccontribs
  \fi
  \endtrivlist
  \endgroup
}
 \makeatother

\title[Closed CMC Hypersurfaces]{On the Rigidity of Closed CMC Hypersurfaces in $\mathbb{S}^5(1)$ with Constant Scalar Curvature}

\author[Deng]{Qintao Deng$^{*}$}
\address[Deng]{School of Mathematics and Statistics \& Key Lab NAA--MOE, Central China Normal University, Wuhan 430079, China}
\email{qintaodeng@ccnu.edu.cn}
\thanks{$^{*}$ Supported in part by NSFC (No. 10901067, No. 12271195) and the Special Fund for Basic Scientific Research of Central Colleges (no. CCNU19QN081).}

\author[Kou]{Yunjia Kou}
\address[Kou]{Department of Mathematics and Statistics, Washington State University, Pullman, WA, 99164-3113, USA}
\email{yunjia.kou@wsu.edu}


\begin{document}

\begin{abstract}
Let $M^4\hookrightarrow\mathbb S^5(1)$ be a closed CMC hypersurface
with constant scalar curvature and constant third power sum
$f_3=\sum_{i,j,k}h_{ij}h_{jk}h_{ki}$. We prove that if $M^4$ has
exactly two distinct principal curvatures at some point, then it is
isoparametric. More precisely, it is a Clifford torus
of the form
$\mathbb S^1(r)\times\mathbb S^3(\sqrt{1-r^2})$ or
$\mathbb S^2(r)\times\mathbb S^2(\sqrt{1-r^2})$, where $0<r<1$. Under the additional Willmore condition, we obtain a complete
classification: every closed CMC Willmore hypersurface
$M^4\hookrightarrow\mathbb S^5(1)$ with constant scalar curvature is
isoparametric. Consequently, it is congruent to a totally umbilic
geodesic sphere, the minimal Clifford torus
$\mathbb S^2(1/\sqrt2)\times\mathbb S^2(1/\sqrt2)$, the nonminimal
Clifford torus
$\mathbb S^1(\sqrt3/2)\times\mathbb S^3(1/2)$, or a Cartan minimal
hypersurface. The proofs combine trace-free local tensor identities,
an algebraic analysis of the possible principal curvature multiplicities,
and a weighted differential $3$-form together with a cut-off argument
near the set where principal curvatures coalesce. No sign condition on
the scalar curvature is imposed.
\end{abstract}


\maketitle

\tableofcontents

\vspace{5mm}

\section{Introduction}

\subsection{From Chern's Conjecture to the CMC Problem}

The starting point is Chern's conjecture for closed minimal
hypersurfaces of codimension one in spheres:
\begin{conjecture}[Chern's conjecture \cite{CdS}]\label{conj:chern}
For each $n$, the set of constant values of $S$ attained by closed
minimal hypersurfaces $M^n\subset\mathbb S^{n+1}(1)$ with constant
scalar curvature is discrete.
\end{conjecture}

Its commonly used strong form predicts that every hypersurface in
Conjecture~\ref{conj:chern} is isoparametric. Simons' fundamental
identity yields the first gap:
\begin{theorem}[Simons \cite{S}]\label{thm:simons-first-gap}
Let $M^n\subset\mathbb S^{n+1}(1)$ be a connected closed minimal
hypersurface. Then
\[
    \int_M S(S-n)\,dM\geq0.
\]
In particular, if $0\leq S\leq n$, then $S\equiv0$ or $S\equiv n$.
\end{theorem}

The equality case $S=n$ was characterized independently by
Chern--do Carmo--Kobayashi \cite{CdS} and Lawson \cite{L}: the
hypersurface is a minimal Clifford hypersurface. The next value suggested by isoparametric theory is $2n$, realized by
Cartan's examples \cite{C},
see also Münzner's general theory \cite{M}. This initiated
the second-gap problem for closed minimal hypersurfaces with constant
$S$. In dimension three, Chang \cite{Ch} completed the classification and
obtained $S\in\{0,3,6\}$. In higher dimensions, the second-gap problem initiated by the estimates of Peng--Terng \cite{PT1,PT3}, was further
developed by Yang--Cheng \cite{YC1,YC2,YC3} and Suh--Yang \cite{SY}.
More recently, Cheng--Wei--Yamashiro \cite{CWY} sharpened the
second-gap estimate for complete minimal hypersurfaces of dimension
$n\geq5$ under the additional assumption that $f_3$ is constant: if
$S>n$, then $S>1.8252n-0.712898$.

\vspace{1mm}

The preceding discussion is specific to hypersurfaces. In higher
codimension, the situation is markedly different. Firester--Tsiamis
\cite{FT} recently disproved the corresponding discreteness conjecture
for all $n\geq3$ and codimension $m\geq4$, and also for even $n\geq4$
and $m\geq3$. On the other hand, Ge--Li--Zhang \cite{GLZ} established
an explicit second-gap rigidity theorem under the additional assumption
that the normal bundle is flat. Since flatness of the normal bundle is
automatic in codimension one, these results further underscore the
distinctive nature of the hypersurface setting.

\vspace{2.7mm}

The corresponding constant mean curvature (CMC) problem is not merely
a change of terminology. If $H$ denotes the normalized mean curvature,
then the assumptions that $H$ and the scalar curvature are constant fix
the first two power sums of the principal curvatures,
\[
    f_1=\sum_{i=1}^n\lambda_i=nH,
    \qquad
    f_2=\sum_{i=1}^n\lambda_i^2=S,
\]
but do not control the higher power sums. This residual algebraic
freedom is absent in dimension three once the characteristic polynomial
is taken into account, but becomes essential in higher dimensions.
Accordingly, two natural ways to recover rigidity are to prescribe
additional power sums or to restrict the number of distinct principal
curvatures. The role of the third power sum $f_3$ in dimension four will
be developed in the next subsection, where it arises geometrically from
the Willmore equation.

The three-dimensional CMC problem was settled completely. de Almeida
and Brito \cite{AB} first proved the result under the additional
assumption that the scalar curvature is positive. Chang then removed
this restriction by a differentiation argument:
\begin{theorem}[Chang \cite{Ch1}]\label{th:Ch1}
Let $M^3\subset\mathbb S^4(1)$ be a closed immersed hypersurface with
constant mean curvature and constant scalar curvature. Then $M^3$ is
isoparametric.
\end{theorem}

In higher dimensions, the first strategy leads to rigidity results
under the constancy of further power sums. Tang--Wei--Yan \cite{TWY}
proved that a closed CMC hypersurface
$M^n\subset\mathbb S^{n+1}(1)$ with nonnegative constant scalar
curvature is isoparametric if $f_3,\ldots,f_{n-1}$ and the number $g$ of
distinct principal curvatures are all constant. Tang--Yan \cite{TYan}
subsequently removed the assumption that $g$ is constant:
\begin{theorem}[Tang--Yan \cite{TYan}]
Let $M^n$, $n>3$, be a closed hypersurface in the unit sphere
$\mathbb S^{n+1}$. Suppose that $R_M\geq0$ and that $\sum_{i=1}^n\lambda_i^k,\  k=1,\ldots,n-1,$
are constant. Then $M^n$ is isoparametric. Moreover, if $M^n$ has $n$
distinct principal curvatures somewhere, then $R_M\equiv0$.
\end{theorem}

The second strategy imposes a direct restriction on the principal
curvature structure. Chang proved isoparametricity when the number of
distinct principal curvatures is $g=3$, without requiring the constancy
of the higher power sums:
\begin{theorem}[Chang \cite{ChangCMC}]\label{th:ChangCMC}
Let $M^n\subset\mathbb S^{n+1}(1)$ be a closed hypersurface with
constant mean curvature and constant scalar curvature. If $M$ has
exactly three distinct principal curvatures at every point, then $M$ is
isoparametric.
\end{theorem}

The closedness assumption in Theorem~\ref{th:ChangCMC} is essential to
Chang's argument, which leads to a genuinely local question. In
dimension three, it is usually formulated as follows:
\begin{conjecture}[Bryant's conjecture]
\label{conj:bryant}
A piece of a minimal hypersurface in $\mathbb S^4(1)$ with constant
scalar curvature is isoparametric.
\end{conjecture}

Chen--Li recently established a higher-dimensional local counterpart
for at most three distinct principal curvatures, thereby extending
Chang's result:
\begin{theorem}[Chen--Li \cite{ChenLiCMC}]\label{th:ChenLiCMC}
Let
$x:M^n\rightarrow\mathbb M^{n+1}(c)$, $n\geq4$, be a piece of an
immersed CMC hypersurface in a space form. If its scalar curvature is
constant and the number $g$ of distinct principal curvatures satisfies
$g\leq3$, then $M^n$ is isoparametric.
\end{theorem}

A complementary recent result reaches the four-principal-curvature
regime by imposing the Dupin condition. Miyaoka \cite{Miyaoka} proved
that a closed proper Dupin CMC hypersurface with $g=4$ and constant
scalar curvature is isoparametric. Since the proper Dupin hypothesis
includes constancy of the principal-curvature multiplicities and
leafwise constancy of each principal curvature, this result is
structurally different from the spectral assumptions used below.

\subsection{The Willmore Condition and the Role of $f_3$}

For a hypersurface $M^n\subset\mathbb S^{n+1}(1)$, the Willmore
functional is
\[
    \int_M\bigl(S-nH^2\bigr)^{n/2}\,d\mathrm{V}.
\]
Li \cite{Li} derived its Euler-Lagrange equation and classified the
isoparametric Willmore hypersurfaces. When both $H$ and $S$ are
constant, the Euler-Lagrange equation reduces to an algebraic relation.
In intrinsic dimension four, it takes the form
\begin{equation}\label{eq:willmore-intro}
    f_3=2HS-4H^3,
\end{equation}
see \cite{Li,LSS}. Thus, in the four-dimensional CMC setting, the
Willmore condition resolves precisely part of the algebraic freedom
identified above: it not only forces $f_3$ to be constant, but also
determines its value in terms of $H$ and $S$.

The minimal case of \eqref{eq:willmore-intro} is $f_3=0$. This case was
first treated under the additional assumption of nonnegative scalar
curvature by Lusala--Scherfner--Sousa \cite{LSS}. Deng--Gu--Wei
subsequently removed the sign condition and obtained a complete
classification:
\begin{theorem}[Deng--Gu--Wei \cite{DGW}]\label{thm:dgw}
Every closed minimal Willmore hypersurface
$M^4\subset\mathbb S^5(1)$ with constant scalar curvature is
isoparametric. In particular, $S\in\{0,4,12\}$.
\end{theorem}

In higher dimensions, Ge--Tan--Yan--Zhang \cite{GTYZ} recently obtained
a quantitative second-gap estimate for closed minimal Willmore
hypersurfaces with constant scalar curvature when $n\geq5$, and
established the conjectural bound $S\geq2n$ under additional lower
bounds on $f_4$. Viewed algebraically, the minimal Willmore condition
singles out the distinguished value $f_3=0$, naturally suggesting the
broader problem of replacing this identity by the weaker requirement
that $f_3$ be constant. In dimension four, Tang--Yang obtained rigidity
under nonnegative scalar curvature together with constancy of the
number of distinct principal curvatures \cite{TY}. Cheng--Li removed
the sign condition while retaining the latter assumption
\cite{ChengLiMin}, whereas Spruck--Xiao removed the constancy assumption
on the number of distinct principal curvatures while retaining
nonnegative scalar curvature \cite{SX}. He--Xu--Zhao recently removed
both remaining restrictions:
\begin{theorem}[He--Xu--Zhao \cite{HXZ}]\label{thm:hxz}
Let $M^4\subset\mathbb S^5(1)$ be a closed minimal hypersurface with
constant scalar curvature and constant $f_3$. Then $M^4$ is
isoparametric. In particular, $S\in\{0,4,12\}$.
\end{theorem}
Other recent rigidity results include the work of Cui \cite{Cui},
Ge--Liu--Luo--Yan \cite{GLLY}, and Deng--Kou \cite{DK} on closed minimal hypersurfaces in
$\mathbb S^5(1)$ with constant scalar and Gauss-Kronecker curvatures, as well as the work of Tao \cite{T} on
$M^5\subset\mathbb S^6(1)$ with constant $S,f_3,f_4$ under suitable
assumptions on the principal curvatures.

\vspace{2.7mm}

The present paper develops a CMC counterpart of
Theorem~\ref{thm:hxz} within the class naturally selected by the
Willmore condition \eqref{eq:willmore-intro}. Whereas
Theorem~\ref{thm:hxz} allows an arbitrary constant value of $f_3$ in
the minimal setting, here $H$ may be nonzero and $f_3$ is determined
by $H$ and $S$. We also isolate a two-principal-curvature branch under
the weaker assumptions that $H$, $S$, and $f_3$ are constant.

\subsection{Main results}
We first establish a rigidity result that does not require the Willmore
condition.

\begin{theorem}\label{th:2root}
    Let $M^4 \hookrightarrow \mathbb{S}^5(1)$ be a closed hypersurface with constant $H, S$, and $f_3$. If there exists a point $P \in M^4$ with exactly two distinct principal curvatures, then $M^4$ is isoparametric. 

    More precisely, $M^4$ must belong to the families of Clifford torus $\mathbb{S}^1(r)\times \mathbb{S}^3(\sqrt{1-r^2})$ or $\mathbb{S}^2(r)\times \mathbb{S}^2(\sqrt{1-r^2})$ with $0<r<1$.
\end{theorem}

Imposing the Willmore relation determines the admissible radii in
Theorem~\ref{th:2root}, yielding the following corollary:

\begin{corollary}\label{cor}
    Let $M^4\hookrightarrow \mathbb{S}^5(1)$ be a closed CMC Willmore hypersurface with constant scalar curvature.
    If $M^4$ has exactly two distinct principal curvatures at some point,
    then $M^4$ is isoparametric and must be one of the following Clifford tori:
    $$ \mathbb{S}^2\left(\frac{\sqrt{2}}{2}\right) \times \mathbb{S}^2\left(\frac{\sqrt{2}}{2}\right) \quad \text{or} \quad \mathbb{S}^1\left(\frac{\sqrt{3}}{2}\right) \times \mathbb{S}^3\left(\frac{1}{2}\right). $$
\end{corollary}

The remaining principal-curvature multiplicity regimes require the
Willmore condition. Combining their analysis with
Corollary~\ref{cor}, we obtain the following classification:

\begin{theorem}\label{th:willmore}
    Let $M^4\hookrightarrow \mathbb{S}^5(1)$ be a closed CMC Willmore hypersurface with constant scalar curvature. Then $M^4$ is isoparametric.
    
    In fact, $M^4$ is congruent to a totally umbilic geodesic sphere $\mathbb{S}^4(r)$ with $0<r\le1$, the minimal Clifford torus $\mathbb{S}^2\left(\frac{\sqrt{2}}{2}\right) \times \mathbb{S}^2\left(\frac{\sqrt{2}}{2}\right)$, the nonminimal Clifford torus $
    \mathbb S^1\left(\frac{\sqrt3}{2}\right)
    \times
    \mathbb S^3\left(\frac{1}{2}\right)$ or a Cartan minimal hypersurface.
\end{theorem}
No sign condition on the scalar curvature or constancy assumption on
the number of distinct principal curvatures is imposed. The proof
combines the two-principal-curvature rigidity result above with a
trace-free reformulation of the weighted differential 3-form
and a cut-off argument adapted from \cite{SX} for the three- and
four-principal-curvature regimes. These arguments reduce the remaining
cases to the known minimal classification. Thus, part of the rigidity
mechanism in the minimal theory persists in a natural nonminimal
Willmore class, while the general CMC problem with constant $f_3$
remains open.

\vspace{2.5mm}

The paper is organized as follows. In Section 2, we
collect the basic formulas for CMC hypersurfaces in the sphere and
develop their trace-free reformulations. In Section 3, we prove
Theorem~\ref{th:2root} by analyzing the possible multiplicity patterns
at a point with exactly two distinct principal curvatures. In
Section 4, under the Willmore condition, we treat the case of four
distinct principal curvatures everywhere using a weighted differential
$3$-form and Stokes' theorem. In Section 5, we analyze the
coalescence of principal curvatures through a cut-off argument, exclude
the three-principal-curvature case, and complete the proof of
Theorem~\ref{th:willmore}.

\section{Basic Formulas}
In this section, we recall some basic formulas for closed CMC hypersurfaces in $\mathbb{S}^{n+1}(1)$, particularly $n=4$, which can be found in \cite{Cui, LSS, PT1, PT3}.

Let $M^n$ be an $n$-dimensional closed hypersurface with constant mean curvature $H$ in $\mathbb{S}^{n+1}(1)$. We choose a local orthonormal frame field $\{e_1, e_2, \cdots, e_{n+1}\}$ such that $\{e_1, e_2, \cdots, e_{n}\}$ is tangent to $M^n$. Let $h_{ij}$ and $H$ denote the components of the second fundamental form and the mean curvature respectively. Then
$$H=\frac{1}{n}\sum_{i}h_{ii}, \quad S=\sum_{i,j}h_{ij}^2, \quad f_3=\sum_{i,j,k}h_{ij}h_{jk}h_{ki}, \quad f_4=\sum_{i,j,k,l}h_{ij}h_{jk}h_{kl}h_{li}.$$

At any fixed point $p\in M$, we can choose an orthonormal frame such that $h_{ij}=\lambda_i\delta_{ij}$ for all $i,j$. Then, at this point $P$, we have
$$H=\frac{1}{n}\sum_{i=1}^n \lambda_i, \quad S=\sum_{i=1}^n \lambda_i^2, \quad f_3=\sum_{i=1}^n \lambda_i^3, \quad f_4=\sum_{i=1}^n \lambda_i^4.$$

The Euler-Lagrange equation for the Willmore functional gives the
following characterization, see \cite{LSS}: Let $M^n \subset \mathbb{S}^{n+1}(1)$ be an $n$-dimensional compact hypersurface
with constant mean curvature and constant scalar curvature in an $(n+1)$-dimensional
unit sphere $\mathbb{S}^{n+1}(1)$. Then $M^n$ is a Willmore hypersurface if and only if
\begin{equation}\label{eq:Willmore}
    f_3 = 2HS-nH^3.
\end{equation}

\vspace{2mm}

We next introduce the trace-free second fundamental form
\[
    \mathring{h} :=h-Hg,\ \ 
    \mu_i=\lambda_i-H.
\]
Notice that in a local orthonormal frame $\{e_i\}_{i=1}^4$, we have 
\[\mathring{h}_{ij} = h_{ij} - H\delta_{ij},\]
where $\delta_{ij}=\begin{cases}1\ \ i=j\\0\ \ i\ne j 
\end{cases}$,
and $\sum_i\mu_i=0$. Correspondingly, we use the notation
\[ \mathring{S} =\sum_{i,j}\mathring{h}_{ij}^2 = \sum_{i=1}^{4}\mu_i^2,\ \ \mathring{f}_k = \sum_{i=1}^{4}\mu_i^k.
\]
Since $H$ is constant,
\[\mathring h_{ijk} = h_{ijk},\ \ \mathring h_{ijkl} = h_{ijkl}
\]
The full and trace-free power sums are related by
\begin{align}
S  &=  \mathring S+4H^2,\\[1.5mm]
    f_3
      &=\mathring f_3+3H\mathring S+4H^3,\label{eq:f3-shift}\\[1.5mm]
    f_4
      &=\mathring f_4+4H\mathring f_3
        +6H^2\mathring S+4H^4,\label{eq:f4-shift}\\[1.5mm]
    f_5
      &=\mathring f_5+5H\mathring f_4
        +10H^2\mathring f_3+10H^3\mathring S+4H^5.
        \label{eq:f5-shift}
\end{align}
In particular, \eqref{eq:Willmore} and \eqref{eq:f3-shift} imply
\begin{equation}\label{eq:f30willmore}
   \mathring{f}_3=(2HS-4H^3) - 3HS+8H^3 = -H\mathring{S}.
\end{equation}

The Simons identity gives \begin{equation}\label{eq:dS}
     \sum_{i,j,k} \mathring{h}_{ijk} ^2 = \mathring{S}(\mathring{S} - 4H^2 - 4) - 4H\mathring{f}_3.
\end{equation}
Define $
    \mathscr{A}      =\sum_{i,j,k=1}^{4}\lambda_i^2 h_{ijk}^2,\ 
    \mathscr{B} =\sum_{i,j,k=1}^{4}\lambda_i\lambda_j h_{ijk}^2$ and their trace-free notion \[
\mathring{\mathscr A} =\sum_{i,j,k=1}^{4}\mu_i^2\mathring h_{ijk}^2,\ \  \mathring{\mathscr B} =\sum_{i,j,k=1}^{4}\mu_i\mu_j \mathring h_{ijk}^2,\] the CMC Simons identity for the higher power sums in
arbitrary intrinsic dimension $n$ is
\begin{equation}\label{eq:df3-general}
    \Delta f_3
      =3(n-S)f_3+3nHf_4-3nHS
       +6\sum_{i,j,k=1}^{n}h_{ijk}^2\lambda_i.
\end{equation} 
In dimension $4$, it gives the equivalent trace-free
version
\begin{equation}\label{eq:df3}
    0=\Delta\mathring f_3
      =3(4-\mathring S+4H^2)\mathring f_3
       +12H\mathring f_4-3H\mathring S^2
       +6\sum_{i,j,k=1}^{4}\mathring h_{ijk}^2\mu_i.
\end{equation}

Similarly, the
dimension-$4$ formula
\begin{equation}\label{eq:df4}
    \frac14\Delta f_4
      =(4-S)f_4+4H(f_5-f_3)+2\mathscr A+\mathscr B,
\end{equation}
and $\mathring f_5=\dfrac{5}{6} \mathring S\mathring f_3$, by virtue of Newton's identities, it implies the trace-free version
\begin{align}\label{eq:df4-tracefree}
    \frac14\Delta\mathring f_4
      &=(4-\mathring S+4H^2)\mathring f_4
       +2\mathring{\mathscr A}+\mathring{\mathscr B}
       -H\bigl(\mathring S\mathring f_3-4\mathring f_5\bigr) \nonumber\\
       &= (4-\mathring S+4H^2)\mathring f_4
       +\frac73H\mathring S\mathring f_3
       +2\mathring{\mathscr A}+\mathring{\mathscr B}.
\end{align}

\section{A Point with Two Distinct Principal Curvatures}
In this section, we prove Theorem~\ref{th:2root}. Assume that $M^4\hookrightarrow \mathbb{S}^5(1)$ has constant $H,S,f_3$ and that, at a point $P\in M^4$, there are two distinct principal curvatures. All computations below are performed at $P$.

The proof of Theorem~\ref{th:2root} is divided into two cases according to the multiplicities of the principal curvatures.

\subsubsection*{Case I. $\lambda_1=\lambda_2 > \lambda_3=\lambda_4$}
Motivated by Theorem 3.1 in \cite{DGW},
we first establish the following lemma.
\begin{lemma}\label{lem:two-eigenvalues-22}
Let $M^4 \hookrightarrow \mathbb S^5(1)$ be a closed CMC hypersurface with
constant $S$ and $f_3$. Suppose that, at the point
$P\in M$,
\[
 \lambda_1=\lambda_2=H+\mu,\qquad
 \lambda_3=\lambda_4=H-\mu,\qquad \mu>0,
\]
then $M^4$ is isoparametric and is a Clifford torus $\mathbb{S}^2 (r) \times \mathbb{S}^2(\sqrt{1-r^2})$, $0<r<1$.
\end{lemma}

\begin{proof}
We consider the trace-free second fundamental form $\mathring{h}_{ij} = h_{ij} - H\delta_{ij}$.
Since $H$ is constant, $h_{ijk}=\mathring{h}_{ijk}$ and $h_{ijkl} = \mathring{h}_{ijkl}$. Throughout this proof, we use $\mathring{\ }$ to denote the trace-free quantities.
By \eqref{eq:dS},
\begin{align*}\sum_{i,j,k}\mathring{h}_{ijk}^2 & = (\mathring{S} + 4H^2)(\mathring{S} + 4H^2-4) +16H^2 - 4H(\mathring{f}_3 + 4H^3 + 3H\mathring{S}) \\[1mm]
& =\mathring{S}(\mathring{S} - 4H^2 -4 ) - 4H\mathring{f}_3.
\end{align*}

In the present case, $\mathring{f}_3=0$. Thus, the following key formulas hold:
\begin{equation}\label{eq:hijko}
    \sum_{i,j,k} \mathring{h}_{ijk}^2 = \mathring{S}(\mathring{S}-4H^2-4)=16\mu^2(\mu^2-H^2-1).
\end{equation}
We first use the identity
\[R_{ijkl}=\delta_{ik}\delta_{jl}-\delta_{il}\delta_{jk}
          +h_{ik}h_{jl}-h_{il}h_{jk}.
\]
Ricci's formula gives, for $i\ne j$,
\begin{equation}\label{eq:app22-comm-general}
 h_{iijj}-h_{jjii}
 =(\lambda_i - \lambda_j)(1+\lambda_i\lambda_j).
\end{equation}

For fixed indices $i\in\{1,2\},\ 
    j\in\{3,4\},$
 we have
\begin{align*}
(\lambda_i-\lambda_j)(1+\lambda_i\lambda_j)
&=2\mu\bigl(1+(H+\mu)(H-\mu)\bigr)\\
&=2\mu(1+H^2-\mu^2).
\end{align*}
Combining this identity with \eqref{eq:hijko}, we obtain
\begin{equation}\label{eq:app22-comm-cross}
    h_{iijj}-h_{jjii}
    =-\frac{\sum_{i,j,k}h_{ijk}^2}{8\mu},
    \ \ 
    i\in\{1,2\},\  j\in\{3,4\}.
\end{equation}
For the reverse ordering, the commutator changes sign as $h_{jjii}-h_{iijj}
    =\frac{\sum_{i,j,k}h_{ijk}^2}{8\mu}.$

The identity needed in place of the minimal identity used in
\cite{DGW} is \eqref{eq:df3}. The key identity is $\big(\sum_{i,j,k,l}h_{ijk}h_{ijl}h_{kl}\big)_m = 0$. Here, however, we have a different formula for $\Delta f_3$, which involves $f_4$. By \eqref{eq:df3} and $\mathring{f}_3=0$, we have
\begin{equation}
    0= \Delta \mathring{f}_3 = 12H\mathring{f}_4 - 3H\mathring{S}^2 + 6
    \sum_{i,j,k,l}
\mathring{h}_{ijk}\mathring{h}_{ijl}\mathring{h}_{kl}.
\end{equation}

At $P$, we have
\[\mathring{f}_4 = 4\mu^4,\ \ \ \mathring{S}=4\mu^2,\]
and hence
\[ \sum_{i,j,k,l} \mathring{h}_{ijk} \mathring{h}_{ijl} \mathring{h}_{kl} =0. \]

To compute $\big(\sum_{i,j,k,l}\mathring{h}_{ijk}\mathring{h}_{ijl}\mathring{h}_{kl}\big)_m = \nabla_m(-2H\mathring{f}_4 + H\mathring{S}^2/2)$, it suffices to consider $\nabla_m \mathring{f}_4$, where
\begin{equation}\label{eq:f4m}
    \nabla_m \mathring{f}_4 = 4 \sum_i \mu_i^3 \mathring{h}_{iim}=4\mu^3(\mathring{h}_{11m} + \mathring{h}_{22m} - \mathring{h}_{33m} - \mathring{h}_{44m}).
\end{equation}
Since $\nabla_m \mathring{S}=0$,
\[\mu\mathring{h}_{11m} + \mu \mathring{h}_{22m} - \mu \mathring{h}_{33m} - \mu  \mathring{h}_{44m}=0,\]
and therefore $\mathring{h}_{11m} + \mathring{h}_{22m} =\mathring{h}_{33m} + \mathring{h}_{44m}$. Substituting this relation into \eqref{eq:f4m} shows that $\nabla_m\mathring{f}_4=0.$
Therefore, we still have $(\sum_{i,j,k,l}\mathring{h}_{ijk}\mathring{h}_{ijl}\mathring{h}_{kl})_m =0$, and Lemma 3.6 in \cite{DGW} remains valid in the trace-free setting.

\vspace{2mm}

Since $H$, $\mathring S$, and $\mathring f_3$ are constant, the
differentiation identities in Section 3 of \cite{DGW} that depend
only on the constancy of the corresponding power sums remain valid
for $\mathring h$, with $\mu_i$ in place of $\lambda_i$. Thus, the same argument applies in the trace-free setting.

With these substitutions, the two tensor alternatives in
\cite{DGW} give, respectively,
\[
-\frac{\sum_{i,j,k}\mathring{h}_{ijk}^2}{32\mu}=\frac{3\sum_{i,j,k}\mathring{h}_{ijk}^2}{32\mu}\  \text{  when} \   \mathring{h}_{113}(P) = 0,\]
or
\[
0=\frac{\sum_{i,j,k}h_{ijk}^2}{8\mu}\ \ \text{when } \ \mathring{h}_{114}(P) =0.
\]

Since $\mu>0$ and $\mathring{S},H$ are constant, it follows that
\[\sum_{i,j,k}h_{ijk}^2=\mathring{S}(\mathring{S}-4H^2-4)\equiv 0,\]
which implies that $h_{ijk}\equiv0$ for $1\le i,j,k \le 4$. Hence, $M^4$ is isoparametric. By the classification in \cite{C}, it is a Clifford torus $\mathbb{S}^2(r)\times\mathbb{S}^2(\sqrt{1-r^2})$ with $0<r<1$.

\end{proof}

\begin{remark}
   (i) Here we use the local component calculation in the proof of
\cite{DGW}, rather than its minimal
classification theorem. The corresponding quantities are replaced as follows:
\[
\begin{array}{c|c}
\hline
\text{minimal} & \text{CMC}\\ \hline
h_{ij} & \mathring{h}_{ij}\\[1mm]
\lambda_i & \mu_i\\[1mm]
f_3\equiv0 & \mathring f_3\equiv0\\[1mm]
\sum h_{ijk}^2=S(S-4) & \sum \mathring{h}_{ijk}^2=S(S-4)+16H^2-4H f_3\\[1mm]
\bigl(\sum h_{ijk}h_{ijl}h_{kl}\bigr)_m=0
& \bigl(\sum \mathring{h}_{ijk} \mathring{h}_{ijl}\mathring{h}_{kl}\bigr)_m=0.\\[1mm]
\hline
\end{array}
\]
The local calculation does not use $\Delta\mathring f_4$.

\vspace{2mm}

(ii) Among the Clifford tori in the family $\mathbb{S}^2(r)\times \mathbb{S}^2(\sqrt{1-r^2})$, only $\mathbb{S}^2(\sqrt{2}/2)\times \mathbb{S}^2(\sqrt{2}/2)$ is Willmore and minimal.

\end{remark}

\vspace{2mm}

\subsubsection*{Case II. $\lambda_1=\lambda_2=\lambda_3=H-\mu < H+3\mu =\lambda_4$.}

\begin{lemma}\label{lem:31}
Let $M^4 \hookrightarrow \mathbb S^5(1)$ be a closed CMC hypersurface with
constant $S$ and $f_3$. Suppose that, at the point
$P\in M$,
\[
 \lambda_1=\lambda_2=\lambda_3=H-\mu,\ \ \lambda_4 = H+3\mu,\ \ \mu >0,
\]
then $M^4$ is isoparametric and is a Clifford torus $\mathbb{S}^3 (r) \times \mathbb{S}^1(\sqrt{1-r^2}),\ 0<r<1$.
\end{lemma}
To prove Lemma~\ref{lem:31}, we use the following elementary inequality.
\begin{lemma}\label{lemma:sqrt3}
    Let $\mu_1,\mu_2,\mu_3,\mu_4\in \mathbb{R}$ with $\mu_1+\mu_2+\mu_3+\mu_4=0$. Then
    \[ -\frac{1}{\sqrt{3}} (\mu_1^2 + \mu_2^2 + \mu_3^2 + \mu_4^2)^{3/2} \le \mu_1^3 + \mu_2^3 + \mu_3^3 + \mu_4^3  \le \frac{1}{\sqrt{3}} (\mu_1^2 + \mu_2^2 + \mu_3^2 + \mu_4^2)^{3/2}.\]
    Equality holds if and only if $\mu_1=\mu_2=\mu_3 \text{ and}\ \mu_4=-3\mu_1$ (up to a permutation of the indices).
\end{lemma}
\begin{proof}
    WLOG, let $\mu_1^2 + \mu_2^2 + \mu_3^2 + \mu_4^2=1$, and define
    $$L =\mu_1^3 + \mu_2^3 + \mu_3^3 + \mu_4^3 - a (\mu_1^2 + \mu_2^2 + \mu_3^2 + \mu_4^2-1) - b (\mu_1+\mu_2+\mu_3+\mu_4).$$
\[\frac{\partial L}{\partial\mu_i} = 3\mu_i^2-2a\mu_i-b=0.\]
Since the polynomial $3t^2-2at -b$ has at most two real roots, the extrema occur when either $\mu_1=\mu_2=\mu_3= \frac{\pm 1}{2\sqrt{3}}, \mu_4=\frac{\mp3}{2\sqrt{3}}$ or $\mu_1=\mu_2=\frac{\pm1}{2}, \mu_3 = \mu_4=\frac{\mp 1}{2}$.
Comparing the corresponding values, we obtain the bounds $-1/\sqrt{3}$ and $1/\sqrt{3}$. Therefore,
\[-\frac{1}{\sqrt{3}} (\mu_1^2 + \mu_2^2 + \mu_3^2 + \mu_4^2)^{3/2} \le \mu_1^3 + \mu_2^3 + \mu_3^3 + \mu_4^3  \le \frac{1}{\sqrt{3}} (\mu_1^2 + \mu_2^2 + \mu_3^2 + \mu_4^2)^{3/2}.\]
\end{proof}

Up to reversing the orientation, we may assume $\mathring{f}_3 >0.$ Since $\mathring{f}_3=24 \mu^3$ and $\mathring{S}=12\mu^2$, a direct calculation gives $\mathring{f}_3 = \frac{1}{\sqrt{3}}\mathring{S}^{3/2}$.

By Lemma~\ref{lemma:sqrt3}, $\mu_1 : \mu_2 : \mu_3 : \mu_4= -1: -1: -1: 3$ everywhere on $M^4$. Since $\mathring{S}=\sum_i\mu_i^2$ is constant, the values of the $\mu_i$ must be $\mu_1=\mu_2=\mu_3=-\mu, \mu_4= 3\mu$ globally.
Thus, the principal curvatures are constant, and $M^4$ is isoparametric. By the classification theorem in \cite{C}, it must be a Clifford torus $\mathbb{S}^3(r)\times\mathbb{S}^1(\sqrt{1-r^2})$ with $0<r<1$.
\begin{remark}
    Among the Clifford tori in the family $\{\mathbb{S}^3(r)\times\mathbb{S}^1(\sqrt{1-r^2}):\ 0<r<1\}$, only $\mathbb{S}^3(1/2)\times \mathbb{S}^1(\sqrt{3}/2)$ is Willmore, and it is nonminimal.
\end{remark}

\vspace{3mm}

\section{Four Distinct Principal Curvatures Everywhere}

In this section, we further assume that $M^4$ is Willmore and prove
Theorem~\ref{th:willmore}.

\begin{lemma}\label{lemma:4root}
Let $M^4 \hookrightarrow \mathbb{S}^5(1)$ be a closed CMC Willmore
hypersurface with constant scalar curvature. If $M^4$ has four distinct
principal curvatures everywhere, then $M^4$ is isoparametric.

In fact, it is minimal and is therefore Cartan's minimal hypersurface.
\end{lemma}


Since $\lambda_1 < \lambda_2 < \lambda_3 < \lambda_4$ everywhere and
$H$, $S$, and $f_3$ are constant, differentiation yields, for each
$k=1,2,3,4$, the system
\[
\begin{cases}
 h_{11k} + h_{22k} + h_{33k} + h_{44k} &= 0,\\[2mm]
 \lambda_1 h_{11k} + \lambda_2 h_{22k}
 + \lambda_3h_{33k} + \lambda_4 h_{44k}&= 0,\\[2mm]
 \lambda_1^2 h_{11k} + \lambda_2^2 h_{22k}
 + \lambda_3^2 h_{33k} + \lambda_4^2 h_{44k}& = 0.
\end{cases}
\]
Because the $\lambda_i$ are distinct, Cramer's rule expresses
$h_{11k}$, $h_{22k}$, and $h_{33k}$ in terms of $h_{44k}$:
\begin{equation}\label{solu}
    \begin{aligned}
        h_{11k} &= -  \frac{(\lambda_3 - \lambda_4)(\lambda_2 - \lambda_4)}{(\lambda_3 - \lambda_1)(\lambda_2 - \lambda_1)} h_{44k},\\[2mm]
        h_{22k} & = - \frac{(\lambda_4 - \lambda_3)(\lambda_4 - \lambda_1)}{(\lambda_2 - \lambda_3)(\lambda_2 - \lambda_1)} h_{44k},\\[2mm]
        h_{33k} &= - \frac{(\lambda_4 - \lambda_1)(\lambda_4 - \lambda_2)}{(\lambda_3 - \lambda_1)(\lambda_3 - \lambda_2)} h_{44k}.
    \end{aligned}
\end{equation}

Passing to the trace-free principal curvatures
$\mu_i = \lambda_i - H$ and using the constancy of $H$, we have
\[
    \lambda_i - \lambda_j = \mu_i - \mu_j,
    \qquad
    h_{ijk} = \mathring{h}_{ijk},
\]
where $\mathring{h}_{ij} = h_{ij} - H\delta_{ij}.$ Consequently, for
$1\le i < j \le 4$,
\[
R_{ijij} = 1 + \lambda_i \lambda_j,\qquad
\omega_{ij}
= \sum_{k=1}^4 \frac{h_{ijk} \omega_k}{\lambda_i - \lambda_j}
= \sum_{k=1}^4
  \frac{\mathring{h}_{ijk} \omega_k}{\mu_i - \mu_j}.
\]

Now we borrow the same 3-form defined as in \cite{SX}:
\begin{definition} \label{theta:Phi}
    Let
\begin{align*}
\theta_{12} &= \omega_3 \wedge \omega_4 \wedge \omega_{12},\  
\theta_{13} = \omega_4 \wedge \omega_2 \wedge \omega_{13}, \ 
\theta_{14} = \omega_2 \wedge \omega_3 \wedge \omega_{14}, \\
\theta_{23} &= \omega_1 \wedge \omega_4 \wedge \omega_{23}, \ \theta_{24} = \omega_3 \wedge \omega_1 \wedge \omega_{24},\ 
\theta_{34} = \omega_1 \wedge \omega_2 \wedge \omega_{34},
\end{align*}
and define the basic 3-form \[
\Phi = \sum_{i<j} (\mu_i + \mu_j) \theta_{ij}.
\]
\end{definition} 
As under the admissible coordinate as in \cite{AB}, \begin{align*}d\Phi &= \sum_{i<j}d(\mu_i + \mu_j)\wedge\theta_{ij} + (\mu_i+\mu_j)\,d\theta_{ij}\\[1mm]
&=\sum_{i<j} \left( \Big(\sum_k(h_{iik} + h_{jjk})\omega_k \wedge\theta_{ij} \Big)  + (\mu_i + \mu_j) \, d\theta_{ij}\right).
\end{align*}
 To compute $d \theta_{ij}$, by virtue of symmetry, it suffices to calculate $d\theta_{12}$ and apply the corresponding permutations. Alternatively, we refer to Lemma 3.7 in \cite{SX} for the complete explicit calculations.
 \begin{proof}[Proof of Lemma \ref{lemma:4root}]
\textbf{Step 1: (computation of $d\Phi$).}
\begin{align}\label{theta12}
d\theta_{12} 
&= (d\omega_3)\wedge \omega_4 \wedge \omega_{12}
- \omega_3 \wedge (d\omega_4)\wedge \omega_{12}
+ \omega_3 \wedge \omega_4 \wedge (d\omega_{12}) \nonumber \\[2mm]
&= \omega_{3k}\wedge \omega_k \wedge \omega_4 \wedge \omega_{12}
- \omega_3 \wedge \omega_{4k}\wedge \omega_k \wedge \omega_{12} + \omega_3 \wedge \omega_4 \wedge \left(\omega_{1k}\wedge \omega_{k2}
- R_{1212}\,\omega_1 \wedge \omega_2\right) \nonumber \\[2mm]
&= \omega_{31}\wedge \omega_1 \wedge \omega_4 \wedge \omega_{12}
+ \omega_{32}\wedge \omega_2 \wedge \omega_4 \wedge \omega_{12}
- \omega_3 \wedge \omega_{41}\wedge \omega_1 \wedge \omega_{12} \nonumber\\[1mm]
&\quad  - \omega_3 \wedge \omega_{42}\wedge \omega_2 \wedge \omega_{12}
+ \omega_3 \wedge \omega_4 \wedge \omega_{13}\wedge \omega_{32} + \omega_3 \wedge \omega_4 \wedge \omega_{14}\wedge \omega_{42} - R_{1212}\mathrm{vol} \nonumber\\[2mm]
& = \vartheta_1 + \vartheta_2 - \vartheta_3 - \vartheta_4 + \vartheta_5 + \vartheta_6 - R_{1212}\mathrm{vol}.    
\end{align}
The forms denoted by $\vartheta_1, \dots, \vartheta_6$ expand as follows:
\begin{align}
\vartheta_1
&=\left(\frac{h_{312}\omega_2 + h_{313}\omega_3}{\mu_3 - \mu_1}\right)
\wedge \omega_1 \wedge \omega_4 \wedge
\left(\frac{h_{122}\omega_2 + h_{123}\omega_3}{\mu_1 - \mu_2}\right)  \nonumber\\[2mm]
& = \left(
\frac{h_{123}^2}{(\mu_3 - \mu_1)(\mu_1 - \mu_2)}
- \frac{h_{331}h_{221}}{(\mu_3 - \mu_1)(\mu_1 - \mu_2)}
\right) \mathrm{vol}.
\end{align}

\begin{align}
\vartheta_2
&= \left(\frac{h_{321}\omega_1 + h_{323}\omega_3}{\mu_3 - \mu_2}\right)
\wedge \omega_2 \wedge \omega_4 \wedge
\left(\frac{h_{123}\omega_3 + h_{121}\omega_1}{\mu_1 - \mu_2}\right) \nonumber \\[2mm]
&= \left(
-\frac{h_{123}^2}{(\mu_3 - \mu_2)(\mu_1 - \mu_2)}
+ \frac{h_{332}h_{112}}{(\mu_3 - \mu_2)(\mu_1 - \mu_2)}
\right) \mathrm{vol}.
\end{align}

\begin{align}
\vartheta_3
&= \omega_3 \wedge 
\left(\frac{h_{412}\omega_2 + h_{414}\omega_4}{\mu_4 - \mu_1}\right)
\wedge \omega_1 \wedge
\left(\frac{h_{124}\omega_4 + h_{122}\omega_2}{\mu_1 - \mu_2}\right) \nonumber\\[2mm]
&= \left(
-\frac{h_{124}^2}{(\mu_4 - \mu_1)(\mu_1 - \mu_2)}
+ \frac{h_{441}h_{221}}{(\mu_4 - \mu_1)(\mu_1 - \mu_2)}
\right)  \mathrm{vol}.
\end{align}
\begin{align}
\vartheta_4
&= \omega_3 \wedge 
\left(\frac{h_{421}\omega_1 + h_{424}\omega_4}{\mu_4 - \mu_2}\right)
\wedge \omega_2 \wedge
\left(\frac{h_{124}\omega_4 + h_{121}\omega_1}{\mu_1 - \mu_2}\right) \nonumber\\[2mm]
&= \left(
\frac{h_{124}^2}{(\mu_4 - \mu_2)(\mu_1 - \mu_2)}
- \frac{h_{442}h_{112}}{(\mu_4 - \mu_2)(\mu_1 - \mu_2)}
\right) \mathrm{vol}.
\end{align}

\begin{align}
\vartheta_5
&= \omega_3 \wedge \omega_4 \wedge
\left(\frac{h_{132}\omega_2 + h_{131}\omega_1}{\mu_1 - \mu_3}\right)
\wedge
\left(\frac{h_{321}\omega_1 + h_{322}\omega_2}{\mu_3 - \mu_2}\right) \nonumber\\[2mm]
&= \left(
-\frac{h_{123}^2}{(\mu_1 - \mu_3)(\mu_3 - \mu_2)}
+ \frac{h_{113}h_{223}}{(\mu_1 - \mu_3)(\mu_3 - \mu_2)}
\right)  \mathrm{vol}.
\end{align}

\begin{align}
\vartheta_6
&= \omega_3 \wedge \omega_4 \wedge
\left(\frac{h_{141}\omega_1 + h_{142}\omega_2}{\mu_1 - \mu_4}\right)
\wedge
\left(\frac{h_{421}\omega_1 + h_{422}\omega_2}{\mu_4 - \mu_2}\right) \nonumber\\[2mm]
&= \left(
-\frac{h_{124}^2}{(\mu_1 - \mu_4)(\mu_4 - \mu_2)}
+ \frac{h_{114}h_{224}}{(\mu_1 - \mu_4)(\mu_4 - \mu_2)}
\right)  \mathrm{vol}.
\end{align}
\vspace{2mm}

By exchanging the indices thoroughly, we can get all the expressions for $d\theta_{ij},\ i,j=1,2,3,4.$ For more details, please see \cite{ChengLiMin} as well. 

\subsubsection*{(i) The Coefficient of $h_{44k}^2$}
Replacing $\lambda_i$ by $\mu_i$ for the same calculation as in \cite{ChengLiMin} the coefficient of $h_{44k}^2$ in $d\Phi$, denoted by $L_k$, is
$$L_k = \mathring{f}_3 \cdot \frac{2(3\mathring{S} - 4\mu_k^2)}{3(\mu_1-\mu_2)^2(\mu_1-\mu_3)^2(\mu_2-\mu_3)^2} =:\mathring{f}_3l_k,$$
Since $\mu_k^2 = (-\sum_{i \neq k} \mu_i)^2 \le 3 \sum_{i \neq k} \mu_i^2 = 3(\mathring{S} - \mu_k^2),$
it shows that $l_k \ge 0$.
\vspace{2mm}


\subsubsection*{(ii) Crossing Terms $h_{ijk}^2$}
\begin{align}\label{crossing}
    & \frac{\mu_i + \mu_j}{(\mu_i-\mu_k)(\mu_j-\mu_k)} - \frac{\mu_i + \mu_k}{(\mu_i-\mu_j)(\mu_j-\mu_k)} + \frac{\mu_j + \mu_k}{(\mu_i-\mu_j)(\mu_i-\mu_k)} \nonumber\\[2mm]
    &= \frac{\left(\mu_i^2-\mu_j^2 - \mu_i^2 + \mu_k^2 + \mu_j^2 -\mu_k^2\right) }{(\mu_i-\mu_k)(\mu_j-\mu_k)(\mu_i-\mu_j)} \nonumber\\[2mm]
    &= 0,
\end{align}

\subsubsection*{(iii) Curvature Term}
$$R_{ijij} = 1 + (\mu_i+H)(\mu_j+H) = 1 + \mu_i\mu_j + H(\mu_i+\mu_j) + H^2$$

\begin{align*}
   \tilde{R} &= \sum_{i<j} (\mu_i + \mu_j)R_{ijij} \\[1.5mm]
   &= (1 + H^2)\sum_{i<j}(\mu_i + \mu_j) + H\sum_{i<j}(\mu_i + \mu_j)^2 + \sum_{i<j}(\mu_i + \mu_j)\mu_i \mu_j\\[1.5mm]
    &= (1+H^2)\cdot 3 \sum_{k=1}^4 \mu_k +H(3\mathring{S}+2\sum_{i<j}\mu_i\mu_j) + \sum_{i\ne j}\mu_i^2\mu_j \\[1.5mm]
    &= 0 + H(3\mathring{S} - \mathring{S}) + \left( \Big(\sum_{i=1}^4 \mu_i^2 \Big) \Big(\sum_{j=1}^4 \mu_j \Big)-\sum_{k=1}^4 \mu_k^3\right)\\[1.5mm]
    &= 2H\mathring{S}-\mathring{f}_3.
\end{align*} 

\vspace{2mm}

\noindent
\textbf{Step 2: (Stokes' Theorem on $M^4$).}
Under the Willmore assumption, substituting
\eqref{eq:f30willmore} into the curvature term $\tilde{R}$ and applying
Stokes' theorem yield
\begin{align*}
    0
    &= \int_{\partial M^4}\Phi
     = \int_{M^4} d\Phi
     = \int_{M^4}
       \left(\sum_{k=1}^4 L_k h_{44k}^2-\tilde{R}\right)
       \mathrm{vol}\\[2mm]
    &= \int_{M^4}
       \left(
       \mathring{f}_3\sum_{k=1}^4 l_k h_{44k}^2
       +\mathring{f}_3-2H\mathring{S}
       \right)
       \mathrm{vol}\\[2mm]
    &= -H\mathring{S}
       \int_{M^4}
       \left(\sum_{k=1}^4 l_k h_{44k}^2+3\right)
       \mathrm{vol}.
\end{align*}
As $l_k\geq0$ for every $k$, 
the integral in the last line is strictly positive. Since
$\mathring{S}\ne0$, the preceding identity forces $H=0$. Hence $M^4$
is minimal. The theorem of Deng--Gu--Wei \cite{DGW} then implies that
$M^4$ is isoparametric. Since it has four distinct principal curvatures
everywhere, it is Cartan's minimal hypersurface.
\end{proof}

\section{Three Distinct Principal Curvatures Somewhere}
\begin{lemma}\label{lemma:no3root}
    Let $M^4 \hookrightarrow \mathbb{S}^5(1)$ be a closed CMC hypersurface with constant scalar curvature. If $M^4$ is Willmore, then $M^4$ cannot have 3 distinct principal curvatures anywhere.
\end{lemma}
Assume that the number of distinct principal curvatures is either 3 or 4 at every point of $M^4$. The two endpoint collision
types
\[
\lambda_1=\lambda_2
\ \ \text{and}\ \ 
\lambda_3=\lambda_4
\]
cannot both occur on $M^4$. After reversing the choice of unit normal
if necessary, we may therefore assume that
$\lambda_1=\lambda_2$ is the only possible endpoint collision. In
particular, $\lambda_3<\lambda_4$ everywhere on $M^4$.
\begin{lemma}\label{lemma:3root}
Suppose that $\mu_1\leq\mu_2\leq\mu_3\leq\mu_4$ and
$\sum_i\mu_i=0$,
\begin{itemize}
  \item[(i)] If $\mu_1=\mu_2<\mu_3<\mu_4$ at $P \in M^4$, then
  $\mathring f_3>0$.
  \item[(ii)] If $\mu_1<\mu_2<\mu_3=\mu_4$ at $P \in M^4$, then
  $\mathring f_3<0$.
  \item[(iii)] If $\mu_1 < \mu_2=\mu_3<\mu_4$ and
  $\mathring f_3>0$ at $P \in M^4$, then $\mu_2=\mu_3<0$.
\end{itemize}
\end{lemma}

\begin{proof}
For (i), since
    \begin{align*}
        \mu_1^3 + \mu_2^3 + \mu_3^3 + \mu_4^3 &= (\mu_1 + \mu_2)^3 - 3\mu_1\mu_2(\mu_1+\mu_2) + (\mu_3+\mu_4)^3 - 3\mu_3 \mu_4(\mu_3+\mu_4)\\[1.5mm]
        &=3(\mu_1+\mu_2)(\mu_3\mu_4 - \mu_1\mu_2),
    \end{align*}
so we only need to consider the sign of $\mu_3\mu_4 - \mu_1\mu_2$.
At $P$, 
\[\mu_3\mu_4<\left(\frac{\mu_3+\mu_4}{2}\right)^2 = \left(\frac{\mu_1+\mu_2}{2}\right)^2 = \mu_1^2,\]
and since $\mathring{f}_3$ is constant on $M^4$, so on $M^4$,
\[\mathring{f}_3 \equiv \mathring{f}_3(P) > 0.\]  (ii) follows from (i) after reversing the orientation.

\vspace{2mm}

To show (iii), if $\mu_2=\mu_3$ at $P$, then
\begin{align}\label{eq:mu2>0}
    0 <\mathring{f}_3 =( \mu_1^3 + \mu_4^3) +2\mu_2^3 &= (-2\mu_2)^3 - 3\mu_1\mu_4\cdot(-2\mu_2) + 2\mu_2^3 \nonumber \\[1.5mm]
    &=6(\mu_1\mu_2\mu_4 - \mu_2^3).
\end{align}
Since 
\[\mu_2^2-\mu_1\mu_4 =\left(\frac{\mu_1 - \mu_4}{2}\right)^2 >0,\]
\eqref{eq:mu2>0} implies $0>\mu_2=\mu_3$.
\end{proof}

In this section, if the cases $\mu_1=\mu_2$ and $\mu_2 = \mu_3$ both happen, by Lemma \ref{lemma:3root}, $\mathring{f}_3>0$. If only one of them happens, then by reversing the orientation, we can always assume $\mathring{f}_3>0$. So we assume $\mathring{f}_3 > 0$ through this section.

\vspace{3mm}

Let $$u =(\lambda_2-\lambda_1)^2=(\mu_2-\mu_1)^2,\  v =(\lambda_3-\lambda_2)^2=(\mu_3-\mu_2)^2,$$
then there exists $\epsilon_0 > 0$ s.t.
\[ u + v \ge 2\epsilon_0.\]

\begin{lemma}
    There exists $\delta>0$ and a constant $C=C(\delta, H,S,f_3)$, s.t. 
    \begin{itemize}
        \item[(i)] on the set $\{0<u<\delta\}$,
        $du \wedge \Phi > C\left(\sum_{k=1}^4h_{44k}^2\right)\mathrm{vol}$;
        \item[(ii)]  on the set $\{0<v<\delta\}$,
        $dv \wedge \Phi > C\left(\sum_{k=1}^4h_{44k}^2\right)\mathrm{vol}$.
    \end{itemize}
\end{lemma}
\begin{proof}
\textbf{Step 1: ($du\wedge\Phi$). } 
\[du \wedge \Phi = \sum_{i=1}^4  u_i\omega_i \wedge \Phi,\]
Denote $\mathcal{W}_{ij}= \mu_i + \mu_j,$ taking the first derivative of $ u=(\mu_2-\mu_1)^2$, we obtain

\begin{align*}
    u_i = e_i(u)&=2(\mu_2 - \mu_1)(h_{22i} - h_{11i})\\[2mm]
    &= 2(\mu_2 - \mu_1) \left[ \frac{(\mu_4 - \mu_3)(\mu_4 - \mu_1)}{(\mu_2 - \mu_1)(\mu_3 - \mu_2)} - \left( -\frac{(\mu_4 - \mu_3)(\mu_4 - \mu_2)}{(\mu_3 - \mu_1)(\mu_2 - \mu_1)} \right) \right] h_{44i} \nonumber \\[2mm]
    &= 2(\mu_4 - \mu_3) \left( \frac{\mu_4 - \mu_1}{\mu_3 - \mu_2} + \frac{\mu_4 - \mu_2}{\mu_3 - \mu_1} \right) h_{44i} =: \tilde{m}_u h_{44i},
\end{align*}
where $\tilde{m}_u >0.$ Substituting into $du\wedge \Phi$, we have 
\begin{align*}
du \wedge \Phi =& \sum_{i=1}^4 u_i  \omega_i \wedge \Phi \nonumber \\[2mm]
=&\  \tilde{m}_u \left[ - \mathcal{W}_{12}\frac{(\mu_4 - \mu_3)(\mu_4 - \mu_1)}{(\mu_3 - \mu_2)(\mu_2 - \mu_1)^2} + \mathcal{W}_{13}\frac{(\mu_4 - \mu_2)(\mu_4 - \mu_1)}{(\mu_3 - \mu_2)(\mu_3 - \mu_1)^2} - \mathcal{W}_{14}\frac{1}{\mu_4 - \mu_1} \right] h_{441}^2 \mathrm{vol} \\[2mm]
& + \tilde{m}_u \left[ \mathcal{W}_{23}\frac{(\mu_4 - \mu_2)(\mu_4 - \mu_1)}{(\mu_3 - \mu_1)(\mu_3 - \mu_2)^2} - \mathcal{W}_{24}\frac{1}{\mu_4 - \mu_2} - \mathcal{W}_{12}\frac{(\mu_4 - \mu_3)(\mu_4 - \mu_2)}{(\mu_3 - \mu_1)(\mu_2 - \mu_1)^2} \right] h_{442}^2 \mathrm{vol} \\[2mm]
& + \tilde{m}_u \left[ - \mathcal{W}_{34}\frac{1}{\mu_4 - \mu_3} + \mathcal{W}_{23}\frac{(\mu_4 - \mu_3)(\mu_4 - \mu_1)}{(\mu_3 - \mu_2)^2(\mu_2 - \mu_1)} - \mathcal{W}_{13}\frac{(\mu_4 - \mu_3)(\mu_4 - \mu_2)}{(\mu_3 - \mu_1)^2(\mu_2 - \mu_1)} \right] h_{443}^2 \mathrm{vol} \\[2mm]
&  + \tilde{m}_u \left[ - \mathcal{W}_{14}\frac{(\mu_4 - \mu_3)(\mu_4 - \mu_2)}{(\mu_3 - \mu_1)(\mu_2 - \mu_1)(\mu_4 - \mu_1)} + \mathcal{W}_{24}\frac{(\mu_4 - \mu_3)(\mu_4 - \mu_1)}{(\mu_2 - \mu_1)(\mu_3 - \mu_2)(\mu_4 - \mu_2)} \right.  \\[2mm]
&\ \ \ \ \ \ \ \ \  \left. - \mathcal{W}_{34}\frac{(\mu_4 - \mu_2)(\mu_4 - \mu_1)}{(\mu_3 - \mu_1)(\mu_3 - \mu_2)(\mu_4 - \mu_3)} \right] h_{444}^2 \mathrm{vol}.
\end{align*}

\vspace{2mm}

 \textbf{Step 2: (singularity near $u=0$)}. Records the full coefficients of
$du\wedge\Phi$. As $u\to0$, the only genuinely second-order
singular terms are
\begin{align*}
 B_{u,1}
 &=-\tilde{m}_u(\mu_1+\mu_2)
 \frac{(\mu_4-\mu_3)(\mu_4-\mu_1)}
 {(\mu_3-\mu_2)(\mu_2-\mu_1)^2},\ \ 
 B_{u,2}
 =-\tilde{m}_u(\mu_1+\mu_2)
 \frac{(\mu_4-\mu_3)(\mu_4-\mu_2)}
 {(\mu_3-\mu_1)(\mu_2-\mu_1)^2}.
\end{align*}
However, since $\mu_1\le \mu_2 \le \mu_3 \le \mu_4$, there exists $\delta_1>0$, s.t. on the set $\{0<u<\delta_1\}$, $\mu_1+\mu_2<0$. Consequently, these second-order singular terms are strictly positive. 

\vspace{3mm}

All other first-order singular
terms cancel in pairs. Indeed, the relevant coefficient of
$h_{443}^2$ contains
\[
 \frac{m_u(\mu_4-\mu_3)}{\mu_2-\mu_1}
 \left[
 \frac{(\mu_2+\mu_3)(\mu_4-\mu_1)}{(\mu_3-\mu_2)^2}
 -\frac{(\mu_1+\mu_3)(\mu_4-\mu_2)}{(\mu_3-\mu_1)^2}
 \right] =: \frac{\mathcal{F}(\mu_1, \mu_2, \mu_3, \mu_4)}{\mu_2 - \mu_1}.
\]
Notice that $\mathcal{F}$ vanishes when $\mu_2=\mu_1$, so it is divisible by
$\mu_2-\mu_1$. As every $\mu_i$ is uniformly bounded on $M^4$, the coefficient of $h_{443}^2$ is actually bounded. The first-order terms multiplying $h_{444}^2$ have the
same divisibility property. Since all other spectral gaps are uniformly
bounded away from zero near this compact collision set, every remainder
is uniformly bounded. This proves (i).

\vspace{2mm}

\textbf{Step 3: (singularity near $\{v=0\}$)}. Taking first derivative of $v=(\mu_3 - \mu_2)^2$, we have 
\begin{align*}
    v_i = e_i(v) &= 2(\mu_3 - \mu_2)(h_{33i} - h_{22i})\\[1.5mm]
    &= 2(\mu_3 - \mu_2)\left( -\frac{(\mu_4 - \mu_1)(\mu_4 - \mu_2)}{(\mu_3 - \mu_1)(\mu_3 - \mu_2)} - \frac{(\mu_4 - \mu_3)(\mu_4 - \mu_1)}{(\mu_3 - \mu_2)(\mu_2 - \mu_1)} \right) h_{44i} \\[1.5mm]
    &=-2(\mu_4 - \mu_1)\left( \frac{\mu_4 - \mu_2}{\mu_3 - \mu_1} + \frac{\mu_4 - \mu_3}{\mu_2 - \mu_1}\right) h_{44i}=:\tilde{m}_v h_{44i},
\end{align*}
where $\tilde{m}_v < 0.$

Similarly, the remaining first-order
singular terms are bounded by the same divisibility argument. As $v \to 0$, the genuinely second-order singular terms multiply
$h_{442}^2$ and $h_{443}^2$, respectively, and their singular parts are
\begin{align*}
 B_{v,2}
 & = \tilde{m}_v(\mu_2+\mu_3)
 \frac{(\mu_4-\mu_2)(\mu_4-\mu_1)}
 {(\mu_3-\mu_1)(\mu_3-\mu_2)^2},\\
 B_{v,3}
 &= \tilde{m}_v(\mu_2+\mu_3)
 \frac{(\mu_4-\mu_3)(\mu_4-\mu_1)}
 {(\mu_2-\mu_1)(\mu_3-\mu_2)^2},
\end{align*}
by Lemma \ref{lemma:3root}, if $v=0$, then $\mathcal{W}_{23}=2\mu_2 <0$. So there exists $\delta_2>0$, s.t. on the set $\{0<v<\delta_2\}$,  $\mathcal{W}_{23} < 0$. Also $\tilde{m}_v <0$, so $B_{v,2}, B_{v,3}$ terms are strictly positive. 

The remaining first-order
singular terms are bounded by the same divisibility argument. This proves (ii).

To make (i) and (ii) be valid simultaneously, we set $\delta = \min\{\delta_1, \delta_2\}$.

\end{proof}

Now we consider $\min\{u,v\}$. To smooth it on $M$, we introduce a smooth even function  $q:\R\to\R$ as in \cite{G}: 
\[
 q(t)\geq|t|,\ \ 
 |q'(t)|\leq1,\ \ 
 q(t)=|t|\ \text{when }|t|\geq \epsilon_0/2,\ \ q''(t)\ge0,
\]
and set
\begin{equation}\label{eq:w}
 w :=\frac{u+v-q(u-v)}2.
\end{equation}
to replace $\min\{u,v\}$.

\vspace{3mm}

 Observe that $u$ and $v$ cannot be identically equal on $M^4$. If so, there will be four distinct principal curvatures everywhere on $M^4$m which returns to Section 4. Therefore the coincidence set $ E := \{u = v\}$ is a closed proper subset of $M^4$.

\vspace{2mm}

Setting $E_{\epsilon_0} := \{ |u - v| \le \epsilon_0\} \supset E$, we have the following properties of $w$:

\begin{itemize}
    \item[(i)] $w = \dfrac{u + v}{2}-\dfrac{|u-v|}{2} = \min\{u, v\}\ge 0$ on $M\setminus E_{\epsilon_0}$. Thus, $w = 0$ iff. $\min\{u,v\} = 0$. 
    \vspace{2mm}
    \item[(ii)] $w\ge \epsilon_0 - \dfrac{\epsilon_0}{2} \ge \dfrac{\epsilon_0}{2}$  on $E_{\epsilon_0}$,
\end{itemize} 

\vspace{5mm}

Let
\[ X = U_\epsilon \sqcup W_\epsilon \sqcup V_\epsilon := \{0 < u < \epsilon\} \sqcup \{w \ge \epsilon\} \sqcup \{0 < v < \epsilon\}. \]
By \cite{C}, there is no 4-dimensional isoparametric hypersurface with $g=3$ in $\mathbb{S}^5(1)$, it ensures that $X\ne\emptyset$.

 \vspace{2mm}
 
For $0 < \epsilon < \min\{\delta, \epsilon_0/2\}$ chosen sufficiently small, we introduce a smooth cut-off function $\eta_\epsilon : \mathbb{R} \to [0,1]$ such that
\begin{itemize}
    \item[(a)] $0 \le \eta_\epsilon \le 1$;

\vspace{1.5mm}
    
    \item[(b)] $\eta_\epsilon(t) = 0$ for $t \le \frac{\epsilon}{3}$;

\vspace{1.5mm}

    \item[(c)] $\eta_\epsilon(t) = 1$ for $t \ge \epsilon$;
    
\vspace{1.5mm}
    
    \item[(d)] $0 \le \eta_\epsilon'(t) \le O \left( \frac{1}{\epsilon} \right)$ for $t \in \left(\frac{\epsilon}{3}, \epsilon\right)$, and $\eta_\epsilon' \equiv 0$ everywhere else.
\end{itemize}

\vspace{2mm}

By Stokes' Theorem and the identity $d\big((\eta_\epsilon \circ w) \Phi\big) = (\eta_\epsilon \circ w)\, d\Phi + (\eta'_\epsilon \circ w)\, dw \wedge \Phi$, we have
\[ \int_X (\eta_\epsilon \circ w)\, d\Phi + \int_X (\eta'_\epsilon \circ w)\, dw \wedge \Phi = 0. \]
Recall that under the assumption of condition, we have $0<\mathring{f}_3 = -H\mathring{S}$,
therefore, 
\begin{align*}
0 &\le  \int_X (\eta_\epsilon \circ w) \left( \sum_{k=1}^4 L_k h_{44k}^2 - \tilde{R} \right) \mathrm{vol} = - \int_X (\eta'_\epsilon \circ w)\, dw \wedge \Phi \\[2mm]
& = - \int_{U_\epsilon} (\eta'_\epsilon \circ u)\, du \wedge \Phi - \int_{V_\epsilon} (\eta'_\epsilon \circ v)\, dv \wedge \Phi \\[2mm]
& \lesssim -\frac{C}{\epsilon}\int_{U_\epsilon \sqcup V_\epsilon} \sum_{k=1}^4 h_{44k}^2\ \mathrm{vol}. 
\end{align*}

Since $\mathring{S}(\mathring{S}-4-4H^2) -4H\mathring{f}_3 = \sum h_{ijk}^2$ is constant, each $h_{ijk}$ is uniformly bounded on $M^4$. Hence, 
\begin{align*}
        |h_{11k}| &= \left|- \frac{(\mu_3 - \mu_4)(\mu_2 - \mu_4)}{(\mu_3 - \mu_1)(\mu_2 - \mu_1)} h_{44k}\right| \quad \text{is bounded on } U_\epsilon;\\[1.5mm]
        |h_{33k}| & = \left|   -\frac{(\mu_4 - \mu_1)(\mu_4 - \mu_2)}{(\mu_3 - \mu_1)(\mu_3 - \mu_2)} h_{44k}\right| \quad \text{is bounded on } V_\epsilon.
\end{align*}
This implies that $|h_{44k}| \lesssim \sqrt{\epsilon}$ on $U_\epsilon \sqcup V_\epsilon$. Consequently, 
\begin{equation}\label{5:calR}
    0 \le \int_X(\eta_\epsilon \circ w) \left( \sum_{k=1}^4 L_k h_{44k}^2 - \tilde{R} \right) \mathrm{vol} \lesssim -C\, \mathrm{vol}(U_\epsilon \sqcup V_\epsilon).
\end{equation}

Notice that $\mathrm{vol}(U_\epsilon) \le \mathrm{vol}(M^4)$ and $\mathrm{vol}(V_\epsilon) \le \mathrm{vol}(M^4)$, both of which are finite. Since $\epsilon > 0$ can be chosen arbitrarily small, taking the limit as $\epsilon \to 0^+$ yields
\[ \lim_{\epsilon\to 0^+} \mathrm{vol}(U_\epsilon) = \mathrm{vol}\left(\bigcap_{\epsilon>0} U_\epsilon\right) = \mathrm{vol}(\emptyset) = 0, \]
\[ \lim_{\epsilon\to 0^+} \mathrm{vol}(V_\epsilon) = \mathrm{vol}\left(\bigcap_{\epsilon>0}V_\epsilon\right) = \mathrm{vol}(\emptyset) = 0. \]
Thus, as $\epsilon \to 0^+$, the inequality \eqref{5:calR} becomes
\[ 0 = \int_{X} \left(\sum_{k=1}^{4} L_k h_{44k}^2 - \tilde{R}\right) \mathrm{vol}. \]
Since $X$ is a nonempty open set, it tells $H\equiv 0$ on $X$. But $H$ is constant on $M^4$, so $M^4$ is minimal. By the result in \cite{DGW}, $M^4$ cannot have 3 distinct principal curvatures anywhere.

\vspace{5mm}

\subsection{Proof of Theorem \ref{th:willmore}}.

\addtocontents{toc}{\protect\setcounter{tocdepth}{-1}} 
We are now at the place to conclude Theorem \ref{th:willmore} by considering the following cases.

\subsubsection*{Case 1}
If there exits one point with no distinct principal curvatures, then $\mathring{S}$ is zero 
and hence $M^4$ is a totally umbilic geodesic sphere $\mathbb{S}^4(r)$ with with $0<r\le1$.

\subsubsection*{Case 2} 
If there exists a point with exactly two distinct principal curvatures, then by Corollary \ref{cor}, $M^4$ is a Clifford torus $\mathbb{S}^2\left(\frac{\sqrt{2}}{2}\right)\times \mathbb{S}^2\left(\frac{\sqrt{2}}{2}\right)$ or $ \mathbb{S}^1\left(\frac{\sqrt{3}}{2}\right) \times \mathbb{S}^3\left(\frac{1}{2}\right)$.

\subsubsection*{Case 3} 
If $M^4$ has exactly four distinct principal curvatures everywhere, then $M^4$ is isoparametric thus is Cartan's minimal hypersurface by Theorem \ref{lemma:4root}. 

\subsubsection*{Case 4} 
If there exists a point with exactly three distinct principal curvatures, this contradicts \ref{lemma:no3root}.

\vspace{2mm}

The proof is complete.

\end{document}